\documentclass[12pt]{amsart}
\usepackage{a4wide}
\usepackage[utf8]{inputenc}
\usepackage{amsmath}
\usepackage{amsthm}
\usepackage{amssymb}
\usepackage{amsfonts}
\usepackage{amsopn}
\usepackage{graphicx}
\usepackage{enumerate}
\usepackage{color}
\usepackage{mathtools}
\usepackage{mathrsfs}
\usepackage{bbm}
\usepackage{yhmath}
\usepackage{cite}
\usepackage[colorlinks,linkcolor=blue,anchorcolor=blue,citecolor=blue]{hyperref}

\newtheorem{theorem}{Theorem}[section]

\newtheorem{lemma}[theorem]{Lemma}

\newtheorem{corollary}[theorem]{Corollary}

\theoremstyle{definition}
\newtheorem{example}{Example}[section]

\newtheorem*{question}{Question}

\numberwithin{equation}{section}

\newcommand{\mbb}{\mathbb}

\newcommand{\mcal}{\mathcal}

\newcommand{\set}[1]{\left\{ #1 \right\}}

\newcommand{\R}{\mathbb{R}}
\newcommand{\Q}{\mathbb{Q}}
\newcommand{\Z}{\mathbb{Z}}
\newcommand{\N}{\mathbb{N}}

\newcommand{\f}{\infty}

\newcommand{\wh}[1]{\widehat{#1}}

\newcommand{\ep}{\varepsilon}
\newcommand{\sse}{\subseteq}
\newcommand{\sm}{\setminus}
\newcommand{\D}{\;\mathrm{d}}

\newcommand{\red}[1]{{\color{red}#1}}

\title[Orthogonal bases of exponential functions for
 infinite convolutions]{Orthogonal bases of exponential functions for
 infinite convolutions
 }

\author[J. J. Miao]{Jun Jie Miao}
\address[J. J. Miao]{School of Mathematical Sciences,  Key Laboratory of MEA(Ministry of Education) \& Shanghai Key Laboratory of PMMP,  East China Normal University, Shanghai 200241, China}
\email{jjmiao@math.ecnu.edu.cn}

\author[H. Zhao]{Hong Bo Zhao}
\address[H. Zhao]{School of Mathematical Sciences,  Key Laboratory of MEA(Ministry of Education) \& Shanghai Key Laboratory of PMMP,  East China Normal University, Shanghai 200241, China}
\email{2504245357@qq.com}

\date{\today}
\subjclass[2010]{28A80, 42C30, 60B10}
\begin{document}

\maketitle

\begin{abstract} 
Let $\mu$ denot the infinite convolution generated by $\{(N_k,B_k)\}_{k=1}^\infty$ given by
$$
\mu =\delta_{{N_1}^{-1}B_1}\ast\delta_{(N_1N_2)^{-1}B_2}\ast\dots\ast\delta_{(N_1N_2\cdots N_k)^{-1}B_k} *\cdots.
$$
where $B_k$ is a complete residue system for each integer $k>0$. We write
$$
\nu_{>k}=\delta_{N_{k+1}^{-1} B_{k+1}} * \delta_{(N_{k+1} N_{k+2})^{-1} B_{k+2}} * \cdots.
$$
Since the elements in $B_k$ may have very large absolute values, the infinite convolution may not be compactly supported. In this paper, we study the necessary and sufficient conditions for such infinite convolutions being a spectral measure.

Generally, for such infinite convolutions, the necessary conditions for spectrality mainly depend on the properties of the polynomials generated by the complete residue systems. The main result shows that if every $B_k$ satisfies uniform discrete zero condition, and $\{\nu_{>k}\}_{k=1}^\infty$ is {\it tight}, then $\# B_k | N_k$ for all integers $k\geq 2$. For some special complete residue systems $\{B_k\}_{k=1}^\infty$, we provide the necessary and sufficient conditions for $\mu$ being a spectral measure.  
\end{abstract}

	\section{Introduction}
	
	A Borel probability measure $\mu$ on $\R^d$ is called a \emph{spectral measure} if there exists a countable subset $\Lambda \sse \R^d$ such that the family of exponential functions
$$
\set{e_\lambda(x) = e^{-2\pi i <\lambda, x>}: \lambda \in \Lambda},
$$
where $<\lambda, x>$ is the standard inner product in $\R^d$,
 forms an orthonormal basis in $L^2(\mu)$, and the set $\Lambda$ is called a \emph{spectrum} of $\mu$. Sometimes, we call $(\mu, \Lambda)$ a {\em spectral pair}.   The  spectrum  existence of $\mu$ is a fundamental problem in harmonic analysis,  and it was initiated by Fuglede in~\cite{ Fuglede-1974}. From then on, it has been studied extensively, see \cite{Dai-2012, Dai-He-Lai-2013, Dai-He-Lau-2014, DJ07, Dutkay-Han-Sun-2014, Dutkay-Han-Sun-Weber-2011, Laba-Wang-2002, Landau67, Matolcsi-2005, Strichartz-2000, Tao-2004 }   and references therein for details.

In 1998, Jorgensen and Pedersen~\cite{Jorgensen-Pedersen-1998} first  discovered that some fractal measures may also have spectrum. A simple example is that the self-similar measure on $\R$ given by the identity
$$
\mu(\;\cdot\;) = \frac{1}{2} \mu(4\;\cdot\;) + \frac{1}{2} \mu(4\;\cdot\; -2)
$$ is a spectral measure with a spectrum
$$
\Lambda = \bigcup_{k=1}^\f \set{\ell_1 + 4\ell_2 + \cdots + 4^{k-1} \ell_k: \ell_1,\ell_2,\cdots,\ell_k \in \set{0,1}}.
$$
From then on, the spectrality and non-spectrality of various singular fractal measures, such as self-affine measures  and Cantor-Moran measures, have been extensively studied, see \cite{An-Fu-Lai-2019,An-He-2014,An-He-Lau-2015,An-He-Li-2015,An-Wang-2021,Dai-He-Lau-2014,Deng-Chen-2021,
Dutkay-Haussermann-Lai-2019,Dutkay-Lai-2014,Dutkay-Lai-2017,Fu-Wen-2017,Laba-Wang-2002,Liu-Dong-Li-2017, Miao-2022} and references therein for details. Readers may refer to~\cite{Falco03} for the details on fractal geometry.

In 2013, He, Lai and Lau proved that a spectral measure $\mu$ must be
of pure type, i.e., $\mu$ is a discrete measure with finite support, absolutely continuous or singularly continuous measure with respect to the Lebesgue measure, see~\cite{HLL} for details.  It implies that there is a big difference between absolutely continuous measures and singular measures in the theory of spectrum. When $\mu$ is the Lebesgue measure restricted on a set $ K$ in $R^d$, it is well-known that the spectral property is closely connected with the tiling
property of $K$, and it is known as the Fuglede conjecture, see~\cite{Fuglede-1974, Kolountzakis-Matolcsi-2006a, Kolountzakis-Matolcsi-2006b,Laba-2001,  Tao-2004 } and references therein for details.

All these fractal spectral measures may be regarded as  special infinite convolutions, and it is natural to investigate the spectrality of  infinite convolutions. 	For a finite subset $A \sse \R^d$, the uniform discrete measure supported on $A$ is given by
\begin{equation*}
\delta_A = \frac{1}{\# A} \sum_{a \in A} \delta_a,
\end{equation*}
where  $\#$ denotes the cardinality of a set and $\delta_a$ denotes the Dirac measure at the point $a$. Given a sequence $\{(N_k,B_k)\}_{k=1}^\infty $, for each integer $k\geq 1$, we write
\begin{equation}\label{def_mun}
\mu_k =\delta_{{N_1}^{-1}B_1}\ast\delta_{(N_1N_2)^{-1}B_2}\ast\dots\ast\delta_{(N_1N_2\cdots N_k)^{-1}B_k},
\end{equation}
where $*$ denotes the convolution between measures. If the sequence $\{\mu_k\}_{k=1}^\infty$ converges weakly to a Borel probability measure $\mu$, then we call $\mu$ the \emph{infinite convolution} of $\{(N_k,B_k)\}_{k=1}^\infty,$ denoted by
\begin{equation}\label{infinite-convolution}
\mu =\delta_{{N_1}^{-1}B_1}\ast\delta_{(N_1N_2)^{-1}B_2}\ast\dots\ast\delta_{(N_1N_2\cdots N_k)^{-1}B_k} *\cdots.
\end{equation}
For each integer $k\geq 1$, we write
\begin{equation}\label{def_mugn}
\mu_{>k}=\delta_{({N_{1}\ldots N_{k+1}})^{-1}B_{k+1}}\ast\delta_{(N_{1}\ldots N_{k+2})^{-1}B_{k+2}}\ast\cdots.
\end{equation}
It is clear that $\mu = \mu_k * \mu_{>k}$.

Admissible pairs are the key to study the spectrality of infinite convolutions. Given a $d \times d$ expansive integral matrix $N$( all eigenvalues have modulus strictly greater than $1$) and a finite subset  $B\sse \Z^d$ with $\# B\ge 2$.  If there exists $L\sse \Z^d$ such that the matrix
$$
\left[ \frac{1}{\sqrt{\# B}} e^{-2 \pi i  (N^{-1}b)\cdot\ell }  \right]_{b \in B, \ell \in L} $$
is unitary, we call $(N, B)$ an {\it admissible pair} in $\R^d$, see~\cite{Dutkay-Haussermann-Lai-2019} for details.

It was first raised by Strichartz~\cite{Strichartz-2000} to study the spectrality of the infinite convolutions based on admissible pairs.
From then on, people focused on find the sufficient conditions for infinite convolutions, that is,
\begin{question}
\emph{Given a sequence of admissible pairs $\{(N_k, B_k)\}_{k=1}^\f$, under what assumptions is the infinite convolution $\mu$ spectral?}
\end{question}

There have been many affirmative partial results obtained for the question, see~\cite{An-Fu-Lai-2019,An-He-2014,An-He-Lau-2015,An-He-Li-2015,Dutkay-Haussermann-Lai-2019,
Dutkay-Lai-2017,Fu-Wen-2017,Laba-Wang-2002} and references therein for details.
If all admissible pairs are identical, i.e., $(N_k,B_k)=(N,B)$ for all $k>0$, then the infinite convolution reduces to a self-affine measure. Dutkay, Haussermann and Lai \cite{Dutkay-Haussermann-Lai-2019} proved that the self-affine measure with equal weights generated by a Hadamard triple in $\R^d$ is a spectral measure, and this result for the one-dimensional case was proved by {\L}aba and Wang \cite{Laba-Wang-2002}.

Instead of finding the sufficient conditions, it is natural to consider the necessary conditions, that is,
\begin{question}
Let the infinite convolution $\mu$ given by \eqref{infinite-convolution} be a spectral measure. What  properties does the sequence $\{(N_k, B_k)\}_{k=1}^\f$ have?
\end{question}
In \cite{Dai-He-Lau-2014}, Dai, He and Lau considered the question for self-similar spectral measures in $\R$,
$$
\mu(\;\cdot\;) = \frac{1}{M}\sum_{j=0}^{M-1} \mu(N(\cdot)-j),
$$
which is a special class of infinite convolutions, i.e., $(N_k, B_k)=(N, B)$ for  all integers $k>0$, where real $N>1$ and $B=\{0,1,\ldots, M-1\}$ for some integer $M\geq 2$. They proved that the self-similar measure $\mu$ is a spectral measure if and only if $N$ is an integer and $M| N$. Recently, Deng, He, Li and Ye  in~\cite{DHLY} studied this question for a class of special infinite convolutions in $R^2$, and they provided a sufficient and necessary condition for the infinite convolution of $\{(R_k,B_k)\}_{k=1}^\infty$ being a spectral measure, where
$$
B_k=\left\{\left(\begin{matrix}
	0\\
	0
\end{matrix}\right),
\left(\begin{matrix}
	1\\
	0
\end{matrix}\right),
\left(\begin{matrix}
	0\\
	1
\end{matrix}\right)\right\},$$
for all $k\in\N$.

Inspired by their work, see~\cite{Dai-He-Lau-2014, DHLY} for details, we study the necessary and sufficient conditions for the spectrality of infinite convolutions $\mu$ generated by a sequence of complete residue systems. Currently, people mainly focused on the spectral infinite convolutions $\mu$ with compact support, but in our setting, the support of infinite convolutions may be unbounded. Our main results are stated in
Section~\ref{sec_M}. In Section~\ref{sec_Pre}, we introduce some basic facts and the key tools used in our work.  In Section \ref{sec_pf}, we analyse the structure of the spectrum of a spectral infinite convolution. The necessary conditions are proved in Section \ref{sec_pfthm}. In Section~\ref{sec_noncpt},  we provide a sufficient and necessary condition for a class special infinite convolutions with unbounded support being a spectral measure.

\section{Complete residue systems and main results}\label{sec_M}

Given an admissible pair $( N,B) $, we call $B$ a {\it general consecutive set}  if
\[
B\equiv \{0,1,2,\cdots, M-1\} \pmod{N},
\]
for some integer $M\geq 2$. Note that $M=\# B$.  Since $( N,B) $ is an admissible pair, it is clear that $M \mid N$. A sequence $\{( N_k,B_k)\}_{k=1}^\infty $ satisfies {\it remainder bounded condition (RBC)} if
\[
\sum_{k=1}^{\infty} \frac{\# B_{k,2}}{\# B_k} < \infty,
\]
where $B_{k,1}=B_k \cap \{0,1,\cdots,N_k-1\}$ and $B_{k,2}=B_k \backslash B_{k,1}.$

Recently, Miao and Zhao in \cite{MZ24} studied the existence and spectrality of infinite convolutions generated by general consecutive sets, and they provided the following sufficient condition for $\mu$ being a spectral measure.

\begin{theorem}\label{thm_MZ}
Let $\{(N_k,B_k)\}_{k=1}^\infty$ be  a sequence  of admissible pairs of general consecutive sets with {\it RBC} satisfied. Then the infinite convolution $\mu$ of $\{(N_k,B_k)\}_{k=1}^\infty$ exists and $\mu$ is a spectral measure.
\end{theorem}

To study the necessary condition, we use a more general definition than general consecutive set, which is frequently used in algebra.
Given a set $B\subset\Z$ and $M\in\N$, we call $B$ a {\it complete residue system} modulo $M$, if
\begin{equation}\label{def_CRS}
B\equiv \{0,1,2,\cdots, M-1\} \pmod{M}.
\end{equation}
Given an admissible pair $(N,B)$, it is clear that if $B$ is a general consecutive set, then $\# B |N$ and $B$ is a complete residue system modulo $\# B$.

Since the elements in the complete residue systems may have very large absolute values, the infinite convolutions generated by a sequence of complete residue systems may not have compact support, see Example \ref{ex_ncpt}. We first provide the sufficient and necessary condition for a special class of infinite convolutions with compact support.

\begin{theorem}\label{consecutive set}
Given a sequence $\{(N_k,B_k)\}_{k=1}^\infty $ such that $B_k=\{0,1,2,\cdots,M_k-1\}$ and $N_k\ge M_k\ge2$ are integers, for all $k\in \N$. Then the infinite convolution $\mu$ is a spectral measure if and only if $M_k|N_k$, for all $k\ge2$.
\end{theorem}

For the first glimpse, it is a little bit of surprise that $N_1$ does not affect the spectrality of the infinite convolution $\mu$. Actually, it is straightforward. Given a real $a\neq 0$, we define a new measure by the infinite convolution $\mu$ such that $\rho(\cdot)=\mu(a\cdot)$. Then we have
\begin{equation*}
\rho
=\delta_{{N'_1}^{-1}B_1}\ast\delta_{(N'_1N_2)^{-1}B_2}\ast\dots\ast\delta_{(N'_1N_2\cdots N_k)^{-1}B_k} *\cdots,
\end{equation*}
where $N'_1=aN_1$. By Corollary \ref{change N},  $(\mu,\Lambda)$ is a spectral pair if and only if $(\rho,a\Lambda)$ is a spectral pair. Therefore the change of  $N_1$ does not affect the spectrality of the infinite convolution $\mu$, and that is the reason for $k\geq 2$ in the Theorem \ref{consecutive set}

The following is an immediate consequence of Theorem \ref{consecutive set}, which was proved by Dai, He and Lau in \cite{Dai-He-Lau-2014}.

\begin{corollary}
Suppose that $(N_k,B_k)=(N,B) $ for all $k\geq 1$ where $B=\{0,1,2,\cdots,M-1\}$ and $N\ge M\ge2$ are integers. Then the infinite convolution $\mu$ is a spectral measure if and only if $M|N$.
\end{corollary}

The complete residue systems $B_k$ used above  are very special, but for general cases, the key is to understand the distribution of zero sets of certain polynomials generated by the complete residue systems.  Given a set $B\subset\Z$, we write
$$
b_{*}=\min\{b:b\in B\},
$$
and  we call
\begin{equation}\label{def_fb}
f_B(x)=\frac{\sum_{b\in B} x^b}{ x^{b_{*}} \sum_{k=0}^{\# B-1}x^k}
\end{equation}
the \textit{character polynomial} of $B$.
Given a function $f:\R\to \R$, we write
\begin{equation}\label{def_ZS}
\mathcal{Z}(f)=\{x\in\R:f(x)=0\}
\end{equation}
for the {\it zero set} of $f$.

Let $B$ be a  complete residue system modulo $M$. We say $B$ satisfies {\it  uniform discrete zero condition (UDZ)} if
\begin{equation}\label{def_UDZ}
\mathcal{Z}(f_B\circ e^{-2\pi i x})\subseteq \Big\{\frac{j}{M}:j\in\Z \setminus M\Z \Big\}.
\end{equation}
Note that there are  complete residue systems which do not  satisfy {\it UDZ}, see Example \ref{e-1}.

Let $\mu$ and $\mu_{>k}$ be an infinite convolution given by \eqref{infinite-convolution} and \eqref{def_mugn}.
We write
\begin{equation}\label{nu_n}
\nu_{>k}(\;\cdot\;) = \mu_{>k}\left( \frac{1}{N_1 N_2 \cdots N_k} \; \cdot\; \right).
\end{equation}
In fact, it is equivalent to
$$
\nu_{>k}=\delta_{N_{k+1}^{-1} B_{k+1}} * \delta_{(N_{k+1} N_{k+2})^{-1} B_{k+2}} * \cdots.
$$

Next, we provide a necessary condition for the infinite convolutions generated by complete residue systems satisfying {\it UDZ}.

\begin{theorem}\label{consecutive digit set}
Given $\{(N_k,B_k)\}_{k=1}^\infty $ where $B_k$ is a  complete residue system modulo $M_k$, and $N_k\ge M_k\ge2$ are integers. Let $\mu$ and $\nu_{>k}$ be given by \eqref{infinite-convolution} and \eqref{nu_n}. Suppose that $B_k$ satisfies uniform discrete zero condition for every $k\geq 1$,  and $\{\nu_{>k}\}_{k=1}^\infty$ is {\it tight}. If $\mu$ is a spectral measure, then $M_k|N_k$, for all $k\ge2$.
	\end{theorem}
	
	We may find some special complete residue systems satisfying {\it UDZ}, and  we have the following conclusion.
	
\begin{corollary}\label{f}
Given a sequence $\{(N_k,B_k)\}_{k=1}^\infty $ such that $B_k$ is a  complete residue system modulo $M_k$ with character polynomial
$$
f_{B_k}(x)=\sum_{i=0}^{2n_k}(-x)^i
$$
satisfying $4n_k+2|M_k$ and $N_k\ge M_k\ge2$ are integers, for all $k\in \N$. Suppose that $\{\nu_{>k}\}_{k=1}^\infty$ is tight. If the infinite convolution $\mu$ is a spectral measure, then $M_k|N_k$, for all $k\ge2$.
\end{corollary}

Finally,  we provide a sufficient and necessary condition for a special class of infinite convolutions. 
Actually, a more general  result is provided in Section \ref{sec_noncpt}, see Theorem \ref{thm_necessary}, and for simplicity, we state the following simple version, which are easy  to construction examples of spectral measures without compact support.   
\begin{theorem}\label{thm_consecutive}
Given $\{(N_k,B_k)\}_{k=1}^\infty $ where $B_k=\{0,1,2,\cdots,M_k-2,n_kM_k-1\}$ and $N_k\ge M_k\ge 3$ and $n_k\geq 0$ are integers for every $k\in \N$. Suppose that $\{M_k\}$ are odd and $\sum_{k=1}^{\f}\frac{1}{M_k}<\f$.  Then $\mu$ is a spectral measure if and only if $M_k|N_k$, for all $k\ge2$.
\end{theorem}

In the end,  we give some examples to illustrate our conclusion. In the first example, we provide a special infinite convolution which has non-compact support, and  the sufficient and necessary condition for spectrality is applied. 

\begin{example}\label{ex_ncpt}
For each $k\geq 1$,  suppose that $N_k\geq (2k+1)^2$ is an integer,
and
$$
B_k=\set{ 0,1,\cdots, (2k+1)^2-2, (2k+1)^2 -1 + (2k+1)^2N_1 N_2 \cdots N_k }.
$$
		
Since  $B_{k,2}=\{(2k+1)^2 -1 + (2k+1)^2N_1 N_2 \cdots N_k\}$, we have
\[
\sum_{k=1}^{\infty} \frac{\# B_{k,2}}{\# B_k} = \sum_{k=1}^{\infty} \frac{1}{(2k+1)^2}<  \infty,
\]
and it implies that  $\{( N_k,B_k)\}_{k=1}^\infty $ satisfies {\it RBC}, by Theorem \ref{existence}, the infinite convolution $\mu$ exists.
And by Theorem \ref{thm_consecutive}, $\mu$ is a spectral measure if and only if $(2k+1)^2|N_k$, for all $k\ge2$.

Since
$$
\sum_{k=1}^{\f} \frac{\max\set{b: b \in B_{k}}}{N_1 N_2 \cdots N_{k}} =\sum_{k=1}^{\f} \left( \frac{(2k+1)^2 -1}{N_1 N_2 \cdots N_{k}} + (2k+1)^2 \right) =\f,
$$
it is clear that  $\mu([0,n])<1$  for all $n\in \N$.  	Therefore  the  measure $\mu$ is not compactly supported.
\end{example}

There exist some complete residue systems which do not satisfy uniform discrete zero condition. In the next example, we provide two complete residue systems with the same character polynomial, but one of them satisfies uniform discrete zero condition, the other does not.

\begin{example}\label{e-1}
Let $B_1=\{0,2,4\}$, $B_2=\{0,2,3,4,5,7\}$. Then $B_1$ is a complete residue system  modulo $3$, and $B_2$ is a  complete residue system modulo $6$. Moreover, their character polynomials are identical, that is,
$$
f_{B_1}(x)=f_{B_2}(x)=1-x+x^2.
$$
Immediately, we have that
$$
\mathcal{Z}(f_{B_1}\circ e^{-2\pi i x})=\mathcal{Z}(f_{B_2}\circ e^{-2\pi i x})=\big\{\pm\frac{1}{6}+k:k\in\Z\big\}.
$$
It is clear that  $B_1$ does not satisfy uniform discrete zero condition, but $B_2$ satisfies  uniform discrete zero condition.

Given a sequence $\{(N_k,B_k)\}_{k=1}^\infty $ such that
$$
B_k=\{2,3,\cdots,M_k-1\}\cup\{0,M_k+1\},
$$
$N_k\ge M_k\ge2$ are integers and $6|M_k$ for all $k\in \N$. It is clear that
$$
f_{B_k}(x)=1-x+x^2.
$$
If the corresponding infinite convolution $\mu$ is a spectral measure, by Corollary \ref{f}, we have that  $M_k|N_k$, for all $k\ge2$.

\end{example}

It happens that infinite convolutions are not spectral measures. In the next example, we
show that the infinite convolutions are not spectral measures even if it is generated by a sequence of complete residue systems.

\begin{example}\label{ex_NE}
For each integer $n\geq 1$, let $N_n =2$ and $B_n=\set{ 0,1 } \text{or} \set{0,3}	.$ It is clear that  $B_n$ is a complete residue system modulo $2$, and $\{( 2,B_n)\}_{n=1}^\infty $ is a sequence of admissible pairs. If
$$
0<\#\{n\in\N:B_n=\set{ 0,1 }\}<\f,
$$
then the infinite convolution $\mu$ is not a spectral measure.

Since $0<\#\{n\in\N:B_n=\set{ 0,1 }\}<\f$, there exists $n_0\in\N$ such that $B_{n_0}=\set{ 0,1 }$ and $B_n=\set{ 0,3}$, for all $n>n_0$.  We have that
$$
\nu_{>n_0-1}=\delta_{\frac{1}{2}\set{ 0,1 }}\ast\delta_{\frac{1}{4}\set{ 0,3 }}\ast\delta_{\frac{1}{8}\set{ 0,3 }}\ast\cdots=\frac{1}{3}\mbb{L}|_{[0,2]}+\frac{1}{3}\mbb{L}|_{[\frac{1}{2},\frac{3}{2}]}.
$$
Hence $\nu_{>n_0-1}$is absolutely continuous with respect to Lebesgue measures, but it is  not uniformly distributed on its support.

Given an admissible pair $(N,B)$, we write $\rho(\cdot)=\nu_{>n_0-1}(N\cdot)$. Obviously, $\delta_{N^{-1}B}\ast\rho$ is still absolutely continuous but not uniform on its support. Thus $\mu$ is absolutely continuous but not uniform on its support.

However, absolutely continuous spectral measures are  uniform distributed on their support, see \cite{Dutkay-Lai-2014} for details. Therefore, $\mu$ is not a spectral measure.
\end{example}

\section{Properties of spectrum and  character polynomials} \label{sec_Pre}
	
We use $\mcal{P}(\R^d)$ to denote the set of all Borel probability measures on $\R^d$.
For $\mu \in \mcal{P}(\R^d)$, the \emph{Fourier transform} of $\mu$ is given by
$$
\wh{\mu}(\xi) = \int_{\R^d} e^{-2\pi i \xi \cdot x} \D \mu(x).
$$
For a set $B\subset \Z$, we write $\mathcal{M}_B(\xi)$ for the Fourier transform of the discrete measure $\delta_B$, that is,
\begin{equation}\label{def_FMB}
\mathcal{M}_B(\xi)=\frac{1}{\# B}\sum_{b\in B}e^{-2\pi i b \xi}.
\end{equation}
Given a Borel probability measure $\mu$ and a subset $\Lambda\subseteq\R$, we write
\begin{align*}
Q_{\mu,\Lambda}(\xi)=\sum_{\lambda\in \Lambda}|\hat{\mu}(\xi+\lambda)|^2,
\end{align*}
if $\Lambda=\emptyset$, we define $Q_{\mu,\Lambda}(\xi)=0$, for all $\xi\in\R$.
The following theorem is often used to verify the spectrality of measures, see~\cite{Jorgensen-Pedersen-1998} for the proofs.
\begin{theorem}\label{Q}\cite{Jorgensen-Pedersen-1998}
Let $\mu$ be a probability measure on $\R$, $\Lambda\subseteq\R$. Then

\noindent (i) the set $\{e^{-2\pi i \lambda\cdot x}:\lambda\in\Lambda\}$ is an orthonormal set in $L^2(\mu)$ if and only if $\hat{\mu}(\lambda_1-\lambda_2)=0$ for all $\lambda_1\ne\lambda_2\in\Lambda$.

\noindent (ii) the set $\{e^{-2\pi i \lambda\cdot x}:\lambda\in\Lambda\}$ is an  orthonormal set in $L^2(\mu)$ if and only if $Q_{\mu,\Lambda}(\xi)\le1$ for all $\xi\in\R$.
		
\noindent (iii) the set $\{e^{-2\pi i \lambda\cdot x}:\lambda\in\Lambda\}$ is an  orthonormal basis in $L^2(\mu)$ if and only if $Q_{\mu,\Lambda}(\xi)\equiv1$ for all $\xi\in\R$.
\end{theorem}

For simplicity, we say $\Lambda$ is {\it an orthogonal set} of $\mu$ if $\{e^{-2\pi i \lambda\cdot x}:\lambda\in\Lambda\}$ is an orthonormal set in $L^2(\mu)$. The conclusion (iii) in the above theorem may be relaxed to verify $Q_{\mu,\Lambda}(\xi)\equiv1$ for all $\xi$ in a dense subset of $\R$.

\begin{corollary}\label{Q'}
Let $\mu$ be a probability measure on $\R$, and let $\Lambda\subseteq\R$ be an orthogonal set of $\mu$. Then the set $\{e^{-2\pi i \lambda\cdot x}:\lambda\in\Lambda\}$ is an orthonormal basis in $L^2(\mu)$ if and only if $Q_{\mu,\Lambda}(\xi)\equiv1$ for all $\xi\in D$, where $D$ is a dense subset of $\R$.
\end{corollary}
	
\begin{proof}
Suppose that $Q_{\mu,\Lambda}(\xi)\equiv1$ for all $\xi\in D$. Then we have
$$
\| e^{2\pi i \xi x}\|=1=Q_{\mu,\Lambda}(\xi)=\sum_{\lambda\in \Lambda}|\hat{\mu}(\xi+\lambda)|^2=\sum_{\lambda\in \Lambda}|\langle e^{-2\pi i \lambda x}, e^{2\pi i \xi x}  \rangle|^2,
$$
for all $\xi\in D$. Since $\Lambda$ is an orthogonal set of $\mu$, we have
$$
\{e^{2\pi i \xi x}:\xi\in D\}\subseteq \overline{\text{span}\{e^{2\pi i \lambda x}:\lambda\in\Lambda\}}.
$$

Since $D$ is a dense subset of $\R$, by Stone-Weierstrass Theorem\cite{Rudin-1991}, we have
$$
\overline{\text{span}\{e^{2\pi i \xi x}:\xi\in D\}}=L^2(\mu),
$$
and it implies that  $L^2(\mu)=\overline{\text{span}\{e^{2\pi i \lambda x}:\lambda\in\Lambda\}}$. Hence $\{e^{-2\pi i \lambda\cdot x}:\lambda\in\Lambda\}$ is  an orthonormal basis in $L^2(\mu)$.
		
The proof of necessity is directly from Theorem \ref{Q} (iii).
\end{proof}

The following two conclusions are direct consequence of Theorem \ref{Q}, which are frequently used in our proofs.
\begin{corollary}\label{prop}
Let $(\mu, \Lambda)$ be a spectral pair. Then we have $(\mu, \Lambda+a)$ is also a spectral pair for all $a\in\R$.
\end{corollary}
As a consequence of above conclusion, we may assume that $0\in\Lambda$  if $\Lambda$ is the spectrum of some probability measure on $\R$.

\begin{corollary}\label{change N}
Given $\mu\in \mcal{P}(\R)$ and a real  $a\neq 0$, let $\rho(\cdot)=\mu(a\cdot)$.Then $(\mu,\Lambda)$ is a spectral pair if and only if $(\rho,a\Lambda)$ is a spectral.
\end{corollary}
\begin{proof}
Since $\rho(\cdot)=\mu(a\cdot)$, we have
$$
\hat{\rho}(a\xi)=\int e^{-2\pi i a\xi x}\D{\rho(x)}=\int e^{-2\pi i a\xi x}\D{\mu(ax)}=\hat{\mu}(\xi),
$$
for all $\xi\in\R$. For every $\lambda\in a\Lambda$, there exists a unique $\lambda'\in \Lambda$ such that $\lambda=a\lambda'$.  For every $\xi\in\R$, write $\xi'=\frac{\xi}{a}$, and we have that
\begin{align*}
\sum_{\lambda\in a\Lambda}\Big|\hat{\rho}(\xi+\lambda)\Big|^2   					=\sum_{\lambda'\in\Lambda}\Big|\hat{\rho}\Big(a\big(\frac{\xi}{a}+\lambda'\big)\Big)\Big|^2
=\sum_{\lambda'\in\Lambda}\Big|\hat{\mu}(\xi'+\lambda')\Big|^2.
\end{align*}

By Theorem \ref{Q},  $(\mu,\Lambda)$ is a spectral pair if and only if $(\rho,a\Lambda)$ is a spectral.
\end{proof}

For each $n\in\N$, we write $\Phi_n(x)$ for the minimal polynomial of $e^{2\pi i \frac{1}{n}}$ over $\Z[x]$. We need the following standard theorem on minimal polynomials to study the properties of character polynomials, see \cite{Washington-1997} for details.
\begin{theorem}\label{minimal polynomial}
For each $q\in\Z$ such that $\gcd(q,n)=1$, $\Phi_n(x)$ is also the minimal polynomial of $e^{2\pi i \frac{q}{n}}$ over $\Z[x]$. Moreover,
\[
\sum_{k=0}^{n-1}x^k=\prod_{p|n,p>1} \Phi_p(x) .
\]
\end{theorem}

The following conclusion shows the connection between a complete residue system modulo $M$ and $\{0,1,\cdots,M-1\}$.

\begin{lemma}\label{function}
Let $B$ be a complete residue system modulo $M$. Then
$$
f_B(x)\in\Z[x].
$$
Moreover, If $B$ satisfies $UDZ$, we have that
$$
\mathcal{Z}(\mathcal{M}_{B})=\Big\{\frac{j}{M}:j\in\Z\setminus M\Z \Big\}  ,
$$
where $\mathcal{M}_{B}$ is the Fourier transform of $\delta_B$ given by \eqref{def_FMB}.
\end{lemma}

\begin{proof}
Since  $B$ be a complete residue system modulo $M$, i.e., $B\equiv \{0,1,2,\cdots, M-1\} \pmod{M}$, for every $l\in\{1,2,\cdots,M-1\}$, we have
\[
\sum_{b\in B}e^{2\pi i \frac{l}{M}\cdot b}=\sum_{k=0}^{M-1}e^{2\pi i \frac{l}{M}\cdot k}=0,
\]
which implies that $e^{2\pi i \frac{1}{M}},\ e^{2\pi i \frac{2}{M}},\cdots, \ e^{2\pi i \frac{M-1}{M}} $ are roots of the polynomial $\sum_{b\in B}x^b$.

Recall that  $b_{*}=\min\{b:b\in B\}$. Since $x^{-b_{*}}\sum_{b\in B}x^b\in\Z[x]$, we have that
$$
\prod_{p|M,p>1} \Phi_p(x)\Big| x^{-b_{*}}\sum_{b\in B}x^b.
$$			
By Theorem \ref{minimal polynomial}, it implies that
$$
\sum_{k=0}^{M-1}x^k\Big| x^{-b_{*}}\sum_{b\in B}x^b.
$$
Recall that $f_B(x)=\frac{x^{-b_{*}}\sum_{b\in B}x^b}{ \sum_{k=0}^{\# B-1}x^k}$, we obtain that $f_B(x)\in\Z[x]$.

Since $f_B(x)\in\Z[x]$ and $B$ satisfies $UDZ$, we have
$$
\mathcal{Z}(\mathcal{M}_{B})=\mathcal{Z}(f_B\circ e^{-2\pi i x})\bigcup\mathcal{Z}( \sum_{k=0}^{\# B-1}e^{-2\pi i kx})= \Big\{\frac{j}{M}:j\in\Z\setminus M\Z \Big\},
$$
and the conclusion holds.
\end{proof}

\section{Spectral decomposition of infinite convolutions }\label{sec_pf}
Recall that the infinite convolution of $\{(N_k,B_k)\}_{n=1}^\infty $ is
\begin{equation*}
\mu =\delta_{{N_1}^{-1}B_1}\ast\delta_{(N_1N_2)^{-1}B_2}\ast\dots\ast\delta_{(N_1N_2\cdots N_k)^{-1}B_k} *\cdots.
\end{equation*}
We write,
\begin{equation}\label{nu}
\nu=\nu_{>1}=\delta_{{N_{2}}^{-1}B_{2}}\ast\delta_{(N_{2}N_{3})^{-1}B_{3}}\ast\cdots,
\end{equation}
thus $\mu(\cdot)=\delta_{{N_1}^{-1}B_1}(\cdot)\ast\nu(N_1\cdot)$ and $\hat{\mu}(\xi)=\mathcal{M}_{B_1}(\frac{\xi}{N_1})\cdot\hat{\nu}(\frac{\xi}{N_1})$.
	
To study the necessary condition for the  spectrality of $\mu$,  we need  to construct the spectrum of $\nu$ by using the spectrum $\Lambda$ of $\mu$ and analyse the structure of the spectrum of $\nu$.
The spectrum of $\mu$ is closely related to the zero set  of its Fourier transform.  To show that the spectrum of $\mu$ is  a  subset of  $\Q$, we have to explore  the  zero set of $\hat{\mu}$. 

We say that the set $\Phi \subseteq \mathcal{P}(\R^d)$ is {\it tight} (sometimes it is also called {\it uniformly tight}) if for each $\epsilon>0$ there exists a compact subset $K\subset\R^d$ such that
\[
\inf_{\eta\in\Phi}\eta(K)>1-\epsilon.
\]
Immediately, we have the following simple facts, and we include the proof here for the readers convenience.

\begin{lemma}\label{equicontinuous}
Let $\Phi \subseteq \mathcal{P}(\R^d)$ be tight. Then the family $\{\hat{\eta}:\eta \in \Phi\}$ is equicontinuous.
\end{lemma}

\begin{proof}
Given $\ep>0$, since $\Phi$ is tight, there exists a compact set $K \sse \R^d$ such that
$$
\inf_{\eta \in \Phi} \eta(K) > 1 - \frac{\ep}{3}.
$$
Then we may find $\delta >0$ such that for all $|y| < \delta$ and all $x\in K$,
$$
\left| 1- e^{2\pi i y\cdot x} \right| < \frac{\ep}{3}.
$$
Hence, for all $\eta \in \Phi$ and all $\xi_1,\xi_2 \in \R^d$ with $|\xi_1 - \xi_2| <\delta$, we have
\begin{align*}
\left| \wh{\eta}(\xi_1) - \wh{\eta}(\xi_2) \right|
    & = \left| \int_{\R^d} e^{-2\pi i \xi_1 \cdot x} \big(1- e^{ 2 \pi i (\xi_1 - \xi_2) \cdot x} \big) \D \eta(x) \right| \\
    & \le \int_K \left| 1- e^{ 2\pi i (\xi_1 - \xi_2) \cdot x} \right| \D \eta(x) + \int_{\R^d \sm K} \left| 1- e^{2\pi i (\xi_1 - \xi_2) \cdot x}  \right| \D \eta(x) \\
    & \le \frac{\ep}{3} \eta(K) + 2 \eta(\R^d \sm K) \\
    &< \ep.
\end{align*}

Therefore, the family $\set{\wh{\eta}: \eta\in \Phi}$ is equicontinuous.
\end{proof}

Let $\mu$ be the infinite convolution of $\{(N_k,B_k)\}_{n=1}^\infty $. Recall that $\mathcal{Z}(\hat{\mu})=\{\xi\in\R:\hat{\mu}(\xi)=0\}$, we may obtain the zero set of $\hat{\mu}$ if $\{\nu_{>k}\}_{k=1}^\infty$ is tight, where $\nu_k$ is given by \eqref{nu_n}.

\begin{lemma}\label{zero set}
Given a sequence $\{(N_k,B_k)\}_{k=1}^\infty $ satisfying $N_k\ge\#B_k\ge2$, for all $k\in \N$. Let $\mu$ be the infinite convolution of $\{(N_k,B_k)\}_{k=1}^\infty $. If $\{\nu_{>k}\}_{k=1}^\infty$ is tight, then
\[
\mathcal{Z}(\hat{\mu})=\bigcup_{k=1}^{\f}N_1N_2\cdots N_k\mathcal{Z}(\mathcal{M}_{B_k}).
\]
\end{lemma}

\begin{proof}	
Since $\{\mu_k\}_{k=1}^\infty$ converges weakly to the infinite convolution $\mu$, we have
$$
\hat{\mu}(\xi)=\lim_{k\rightarrow\f}\hat{\mu}_k(\xi)=\prod_{k=1}^{\f}\mathcal{M}_{B_k}((N_1N_2\cdots N_k)^{-1}\xi),
$$
and it implies that
$$
\bigcup_{k=1}^{\f}N_1N_2\cdots N_k\mathcal{Z}(\mathcal{M}_{B_k})\subseteq \mathcal{Z}(\hat{\mu}).
$$

It remains to show that $\bigcup_{k=1}^{\f}N_1N_2\cdots N_k\mathcal{Z}(\mathcal{M}_{B_k})\supseteq \mathcal{Z}(\hat{\mu}).$  Since $\{\nu_{>k}\}_{k=1}^\infty$ is tight, by Lemma \ref{equicontinuous}, the set $\{\hat{\nu}_{>k}\}_{k=1}^\infty$ is equicontinuous. Note that $\hat{\nu}_{>k}(0)=1$  for all $k\in\N$, the equicontinuity of $\{\hat{\nu}_{>k}\}_{k=1}^\infty$ implies that there exists $\delta>0$ and $\epsilon>0$ such that
$$
\hat{\nu}_{>k}(\xi)>\epsilon,
$$
for all $|\xi|<\delta$ and $k\in\N$. Hence for each given  $\xi'\in \mathcal{Z}(\hat{\mu})$, there exists an integer $K>0$ such that 	
$$
|(N_1N_2\cdots N_K)^{-1}\xi'|<\delta,
$$
and this implies that
$$
\hat{\nu}_{>K-1}\big((N_1N_2\cdots N_K)^{-1}\xi'\big)>\epsilon.
$$	
Immediately, it follows that
\begin{align*}
\hat{\mu}(\xi')	&=\hat{\nu}_{>m-1}\big((N_1N_2\cdots N_m)^{-1}\xi'\big)\prod_{k=1}^{m-1}\mathcal{M}_{B_k}\big((N_1N_2\cdots N_k)^{-1}\xi'\big)\\
&\geq \epsilon\prod_{k=1}^{m-1}\mathcal{M}_{B_k}\big((N_1N_2\cdots N_k)^{-1}\xi'\big),
\end{align*}
Since  $\xi'\in \mathcal{Z}(\hat{\mu})$, i.e., $\hat{\mu}(\xi')=0$, we have that
$$
\xi'\in\bigcup_{k=1}^{m-1}N_1N_2\cdots N_k\mathcal{Z}(\mathcal{M}_{B_k}).
$$

Therefore
$$
\mathcal{Z}(\hat{\mu})\subseteq\bigcup_{k=1}^{\f}N_1N_2\cdots N_k\mathcal{Z}(\mathcal{M}_{B_k}),
$$
and we complete the proof.
\end{proof}

 In the rest of this section, the following assumption is need to  decompose the spectrum of infinite convolution $\mu$.

\textbf{Assumption *}: {\em Given $\{(N_k,B_k)\}_{k=1}^\infty $ where $B_k$ is a complete residue system modulo $M_k$ and satisfies uniform discrete zero condition(UDZ) and $N_k\ge M_k\ge2$ are integers, for all $k\in \N$, let $\mu$ and $\nu_{>k}$ be the corresponding measures given by~\eqref{infinite-convolution} and  \eqref{nu_n}, respectively.  Assume that $\{\nu_{>k}\}_{k=1}^\infty$ is tight.}

The next conclusion show that all elements in the spectrum of $\mu$ are rational numbers.
\begin{lemma}\label{rational number}
Under Assumption *, Let   $\Lambda$ be the spectrum of the infinite convolution $\mu$ with $0\in\Lambda$. Then $\Lambda\subseteq\Q$
\end{lemma}

\begin{proof}
For every $k\in \N$, since $B_k$ is a complete residue system modulo $M_k$ satisfying UDZ, by Lemma  \ref{function}, we have
\begin{eqnarray*}
\mathcal{Z}(\mathcal{M}_{B_k})
=\Big\{\frac{j}{M_k}:j\in\Z\setminus M_k\Z \Big\}
\subseteq \Q.
\end{eqnarray*}
Since $\{\nu_{>k}\}_{k=1}^\infty$ is tight, by Lemma \ref{zero set}, we have
\[
\mathcal{Z}(\hat{\mu})=\bigcup_{k=1}^{\f}N_1N_2\cdots N_k\mathcal{Z}(\mathcal{M}_{B_k}).
\]
Immediately, it follows that
\[
\mathcal{Z}(\hat{\mu})\subseteq\Q.
\]
	
Since $(\mu, \Lambda)$ is a spectral pair, by Theorem ~\ref{Q} (i), we have that $\hat{\mu}(\lambda_1-\lambda_2)=0$, for $\lambda_1\ne\lambda_2 \in \Lambda$. Since $0\in\Lambda$, and it immediately follows that
$$
\Lambda\setminus \{0\}\subseteq\mathcal{Z}(\hat{\mu})\subseteq\Q,
$$
and the conclusion holds
\end{proof}
	
Next,  we show  that the spectrum $\Lambda$ may be decomposed into countable pairwise disjoint sets.
Let $(N_1,B_1)$ be the first term in the sequence $\{(N_k,B_k)\}_{k=1}^\infty$.  We write
$$
\Gamma_0=\Big\{0,\frac{N_1}{M_1},\frac{2N_1}{M_1},\cdots,\frac{(M_1-1)N_1}{M_1}\Big\} .
$$
Note that $(\delta_{{N_1}^{-1}B_1},\Gamma_0)$ is a spectral pair since the dimension of $L^2(\delta_{{N_1}^{-1}B_1})=M_1$ and $\Gamma_0$ is the orthogonal set of $\delta_{{N_1}^{-1}B_1}$.
Given a subset $\Lambda\subseteq\Q$. For every $\gamma\in [0,N_1)\cap\Q$,  we write
\begin{equation}\label{def_PG}
P(\gamma)=\{\omega\in\Z:\gamma+N_1\omega\in\Lambda\},
\end{equation}
 and  we have the following decomposition
\begin{equation}\label{eqn_LDC}
\Lambda=\bigcup_{\gamma\in [0,N_1)\cap\Q}\big(\gamma+N_1P(\gamma)\big).
\end{equation}

Unfortunately, this decomposition is not good enough for our purpose, and we have to make a further decomposition of the set $[0,N_1)\cap\Q$. Since the set $[0,\frac{N_1}{M_1})\cap\Q$ is  countable,  we  assume that
\begin{equation}\label{def_gam}
\Gamma=\Big[0,\frac{N_1}{M_1}\Big)\cap\Q=\{0,\gamma_1,\gamma_2,\cdots\}.
\end{equation}

For each  $k\ge1$, $\gamma_k\in \Gamma$, we write
$$
\ \Gamma_k=\gamma_k+\Gamma_0.
$$
Moreover, we have
\begin{equation}\label{eqn_DNQ}
[0,N_1)\cap\Q=\bigcup_{k=0}^\f\Gamma_k.
\end{equation}

We define that
$$
I^\f=\Gamma_0\times \Gamma_{1}\times \Gamma_{2}\times\cdots.
$$
For $\mathbf{i}=(i_0,i_{1},i_{2},\cdots)\in I^\f$, we write
\begin{equation}\label{def_Li}
\Lambda_\mathbf{i}=\bigcup_{k=0}^\f\Big(\frac{i_k}{N_1}+P(i_k)\Big).
\end{equation}
We are ready to use this construction to show that $\Lambda_\mathbf{i}$ is  the spectrum of $\nu$.

To show $(\nu,\Lambda_\mathbf{i})$ is a spectral pair, we need to analyse the othogonality of the sets $P(\gamma)$ and their translations.
	
\begin{lemma}\label{orthogonal 1}
Under Assumption *, suppose that $(\mu, \Lambda)$ is a spectral pair with $0\in\Lambda$. Then for each $\gamma\in[0,N_1)\cap\Q$, $P(\gamma)$ is either an empty set or an orthogonal set of $\nu$.
\end{lemma}

\begin{proof}
It is sufficient to show that $P(\gamma)$ is orthogonal if $P(\gamma)\ne\emptyset$. For $\alpha\ne\beta\in P(\gamma)\subseteq\Z$, we have
$\gamma+N_1\alpha,\gamma+N_1\beta\in\Lambda$. By Theorem~\ref{Q} (i), this implies that
$$
\hat{\mu}(\gamma+N_1\alpha-\gamma-N_1\beta) =0.
$$
On the other hand, since $\alpha$ and $\beta$ are integers, it follows that
\begin{align*}
\hat{\mu}(\gamma+N_1\alpha-\gamma-N_1\beta)=\mathcal{M}_{B_1}(\alpha-\beta)\cdot\hat{\nu}(\alpha-\beta)=\hat{\nu}(\alpha-\beta).
\end{align*}
Hence  by Theorem~\ref{Q} (i), $\hat{\nu}(\alpha-\beta)=0$ implies that  $P(\gamma)$ is an orthogonal set of $\nu$.
\end{proof}

\begin{lemma}\label{orthogonal 2}
Under Assumption *, suppose that $(\mu, \Lambda)$ is a spectral pair with $0\in\Lambda$. Then given $k\ne j$, for each $i_k\in \Gamma_k$ and each $i_j\in \Gamma_j$,   $\big(\frac{i_k}{N_1}+P(i_k)\big)\bigcup\big(\frac{i_j}{N_1}+P(i_j)\big)$ is either an empty set or an orthogonal set of $\nu$.
\end{lemma}

\begin{proof}
If  one of $P(i_k)$ and $P(i_j)$ is an empty set, by Lemma \ref{orthogonal 1}, the set $\big(\frac{i_k}{N_1}+P(i_k)\big)\bigcup\big(\frac{i_j}{N_1}+P(i_j)\big)$ is either an empty set or an orthogonal
set of $\nu$, and the conclusion holds.  It remains to show the conclusion holds if both of  $P(i_k)$ and $P(i_j)$ are not  empty sets.
		
For $\alpha\in P(i_k)\subseteq\Z$, $\beta\in P(i_j)\subseteq\Z$, by Theorem \ref{Q} (i), we have that $
\hat{\mu}(i_k+N_1\alpha-i_j-N_1\beta)=0$ and
\begin{align*}
\hat{\mu}(i_k+N_1\alpha-i_j-N_1\beta)
&=\mathcal{M}_{B_1}\Big(\frac{i_k}{N_1}+\alpha-\frac{i_j}{N_1}-\beta\Big)\cdot\hat{\nu}\Big(\frac{i_k}{N_1}+\alpha-\frac{i_j}{N_1}-\beta\Big)\\
&=\mathcal{M}_{B_1}\Big(\frac{i_k}{N_1}-\frac{i_j}{N_1}\Big)\cdot \hat{\nu}\Big(\frac{i_k}{N_1}+\alpha-\frac{i_j}{N_1}-\beta\Big).
\end{align*}
It is sufficient to show that $\mathcal{M}_{B_1}(\frac{i_k}{N_1}-\frac{i_j}{N_1})\neq 0$, which implies that $\hat{\nu}(\frac{i_k}{N_1}+\alpha-\frac{i_j}{N_1}-\beta)=0$, that is $\big(\frac{i_k}{N_1}+P(i_k)\big)\bigcup\big(\frac{i_j}{N_1}+P(i_j)\big)$ is  an orthogonal set of $\nu$.

Since $i_k\in \Gamma_k$ and $i_j\in \Gamma_j$, there exists $t_k, t_j\in\{0,1,\cdots,M_1-1\}$ such that
$$
i_k=\gamma_k+t_k\frac{N_1}{M_1}	\ \text{and} \ i_j=\gamma_j+t_j\frac{N_1}{M_1},
$$
where $\gamma_k,\gamma_j\in\Gamma$.
It is clear that
$$
0<\Big|\frac{\gamma_k-\gamma_j}{N_1}\Big|<\frac{1}{M_1},
$$
and it follows that
$$
\frac{i_k}{N_1}-\frac{i_j}{N_1}=\frac{\gamma_k-\gamma_j}{N_1}+\frac{t_k-t_j}{M_1}\in \bigg(\frac{t_k-t_j-1}{M_1},\frac{t_k-t_j}{M_1}\bigg)\bigcup \bigg(\frac{t_k-t_j}{M_1},\frac{t_k-t_j+1}{M_1}\bigg),
$$

Recall that  $B_1$ is a complete residue system modulo $M_1$ satisfying {\it UDZ}, by Lemma~\ref{function}, we have
$$
\frac{i_k}{N_1}-\frac{i_j}{N_1}\notin \Big\{\frac{j}{M_1}:j\in\Z\setminus M_1\Z\Big\}=\mathcal{Z}(\mathcal{M}_{B_1}).
$$
This  implies that $\mathcal{M}_{B_1}(\frac{i_k}{N_1}-\frac{i_j}{N_1})\neq 0$,
and the conclusion holds.
\end{proof}
Lemma \ref{orthogonal 1} and Lemma \ref{orthogonal 2} immediately imply the following conclusion.	
\begin{corollary}\label{cor_nuLo}
Under Assumption *,  for each $\mathbf{i}=(i_0,i_{1},i_{2},\cdots)\in I^\f$, the set $\Lambda_\mathbf{i}$ is either an empty set or an orthogonal set of $\nu$.
\end{corollary}	

To show $(\nu,\Lambda_\mathbf{i})$ is a spectral pair,  it remains  to show  the  completeness of $\Lambda_\mathbf{i}$.
For each $\gamma\in[0,N_1)\cap\Q$, we write
\begin{equation}\label{def_pqg}
p_\gamma(\xi)=\bigg|\mathcal{M}_{B_1}\bigg(\frac{\xi+\gamma}{N_1}\bigg)\bigg|^2, \qquad q_\gamma(\xi)=\sum_{\omega\in P(\gamma)}\bigg|\hat{\nu}\bigg(\frac{\xi+\gamma}{N_1}+\omega\bigg)\bigg|^2.
\end{equation}

\begin{theorem}\label{new spectrum}
Under Assumption *, let $(\mu, \Lambda)$ be a spectral pair with $0\in\Lambda$.  Then for each $\mathbf{i}=(i_0,i_{1},i_{2},\cdots)\in I^\f$, $(\nu,\Lambda_\mathbf{i})$ is a spectral pair.
\end{theorem}
\begin{proof}
Since $(\mu,\Lambda)$ is a spectral pair, by Theorem \ref{Q} (iii),
$$
Q_{\mu,\Lambda}(\xi)=\sum_{\lambda\in \Lambda}|\hat{\mu}(\xi+\lambda)|^2=1.
$$
Moreover, by Lemma \ref{rational number}, we have that $\Lambda\subseteq\Q$. By \eqref{eqn_LDC} and \eqref{eqn_DNQ}, we have
\begin{align*}
\sum_{\lambda\in \Lambda}|\hat{\mu}(\xi+\lambda)|^2
&=\sum_{\gamma\in[0,N_1)\cap\Q}\sum_{\omega\in P(\gamma)}|\hat{\mu}(\xi+\gamma+N_1\omega)|^2\\
&=\sum_{\gamma\in[0,N_1)\cap\Q}\sum_{\omega\in P(\gamma)}\bigg|\mathcal{M}_{B_1}\bigg(\frac{\xi+\gamma}{N_1}+\omega\bigg)\bigg|^2\bigg|\hat{\nu}\bigg(\frac{\xi+\gamma}{N_1}+\omega\bigg)\bigg|^2\\
&=\sum_{\gamma\in[0,N_1)\cap\Q}\bigg|\mathcal{M}_{B_1}\bigg(\frac{\xi+\gamma}{N_1}\bigg)\bigg|^2\sum_{\omega\in P(\gamma)}\bigg|\hat{\nu}\bigg(\frac{\xi+\gamma}{N_1}+\omega\bigg)\bigg|^2\\
&=\sum_{k=0}^\f\sum_{\gamma\in \Gamma_k}\bigg|\mathcal{M}_{B_1}\bigg(\frac{\xi+\gamma}{N_1}\bigg)\bigg|^2\sum_{\omega\in P(\gamma)}\bigg|\hat{\nu}\bigg(\frac{\xi+\gamma}{N_1}+\omega\bigg)\bigg|^2,
\end{align*}
for all $\xi\in\R$.

By~\eqref{def_pqg}, we have that
\begin{equation}\label{eqn_p1}
\sum_{k=0}^\f\sum_{\gamma\in\Gamma_k}p_{\gamma}(\xi)q_{\gamma}(\xi)=\sum_{\lambda\in \Lambda}|\hat{\mu}(\xi+\lambda)|^2= 1,
\end{equation}
for all $\xi\in\R$.
Since $(\delta_{{N_1}^{-1}B_1},\Gamma_0)$ is a spectral pair, by Corollary \ref{prop}, we have $(\delta_{{N_1}^{-1}B_1},\Gamma_k)$ is a spectral pair, for every integer $k\ge0$.  By Theorem \ref{Q} (iii), we have
\[
\sum_{\gamma\in\Gamma_k}p_{\gamma}(\xi)=1,
\]
for all $\xi\in\R$. Since $B_1$ is a complete residue system modulo $M_1$ satisfying {\it UDZ}, by Lemma \ref{function}, the zeros of $\mathcal{M}_{B_1}$ are all rational numbers. By choosing  irrational  $\xi$, we have  that $p_{\gamma}(\xi)>0$ for all $\gamma\in[0,N_1)\cap\Q$.

Fix $\xi_0 \in\R$. For k=0,1,2\ldots, we choose $j_k\in \Gamma_k$ such that
\begin{equation}\label{def_qmax}
q_{j_k}(\xi_0)=\max_{\gamma\in \Gamma_k}q_{\gamma}(\xi_0).
\end{equation}
Let $\mathbf{j}=(j_0,j_1,\ldots ,j_k,\ldots) $ where $j_k$ is given by \eqref{def_qmax}. It is clear that $\mathbf{j}\in I^\f$.  By Corollary \ref{cor_nuLo}, $\Lambda_\mathbf{i}$ is either an empty set or orthogonal set of $\nu$, for all $\mathbf{i}\in I^\f$, hence $\Lambda_{\mathbf{j}}$ is an orthogonal set of $\nu$. By Theorem \ref{Q} (ii), this implies that
$$
0\leq \sum_{\lambda\in \Lambda_{\mathbf{j}}} \bigg|\hat{\nu}\bigg(\frac{\xi_0}{N_1} +\lambda\bigg)\bigg|^2\leq 1,
$$
and it is equivalent to
\[
0\le\sum_{k=0}^\f\max_{\gamma\in \Gamma_k}q_{\gamma}(\xi_0)\le1.
\]
Thus for each $\xi\in\R \setminus \Q$,
\begin{equation}\label{1}
\begin{split}
\sum_{k=0}^\f\sum_{\gamma\in\Gamma_k}p_{\gamma}(\xi)q_{\gamma}(\xi)
&\le\sum_{k=0}^\f\sum_{\gamma\in\Gamma_k}p_{\gamma}(\xi)\max_{\gamma\in\Gamma_k}q_{\gamma}(\xi)\\
&=\sum_{k=0}^\f\max_{\gamma\in \Gamma_k}q_{\gamma}(\xi)\\
&\le1,
\end{split}
\end{equation}
Since $\sum_{k=0}^\f\sum_{\gamma\in\Gamma_k}p_{\gamma}(\xi)q_{\gamma}(\xi)= 1,$  by \eqref{eqn_p1}, it follows that
\begin{align}\label{2}
\sum_{k=0}^\f\max_{\gamma\in \Gamma_k}q_{\gamma}(\xi)=1,
\end{align}
for all $\xi\in\R \setminus \Q$.	
Since $p_{\gamma}(\xi)>0$ for all $\xi\in\R\setminus\Q$ and all $\gamma\in[0,N_1)\cap\Q$, by \eqref{1}, we also have
\begin{align}\label{3}
q_{\gamma'}(\xi)=\max_{\gamma\in\Gamma_k}q_{\gamma}(\xi),
\end{align}
for all $\gamma'\in\Gamma_k$, all $k\ge0$ and all $\xi\in\R \setminus \Q$.

To show that for every $\mathbf{i}\in I^\f$, $(\nu,\Lambda_\mathbf{i})$ is a spectral pair, it is equivalent to show that
$$
\sum_{\lambda\in\Lambda_\mathbf{i}}\bigg|\hat{\nu}\bigg(\frac{\xi}{N_1}+\lambda\bigg)\bigg|^2=1.
$$
Combining \eqref{def_Li}, \eqref{2} and \eqref{3} with  Lemma \ref{orthogonal 1} and Lemma \ref{orthogonal 2} together, we have
\begin{eqnarray*}
\sum_{\lambda\in\Lambda_\mathbf{i}}\bigg|\hat{\nu}\bigg(\frac{\xi}{N_1}+\lambda\bigg)\bigg|^2 &=& \sum_{\lambda\in \bigcup_{k\in\N}\big(\frac{i_k}{N_1}+P(i_k)\big)}\bigg|\hat{\nu}\bigg(\frac{\xi}{N_1}+\lambda\bigg)\bigg|^2 \\
&=& \sum_{k=0}^\f\sum_{\omega\in P(\gamma)}\Big|\hat{\nu}\Big(\frac{\xi}{N_1}+\frac{i_k}{N_1}+\omega\Big)\Big|^2   \\
&=&\sum_{k=0}^\f q_{i_k}(\xi)  \\
&=&1,
\end{eqnarray*}
for all $\xi\in\R \setminus \Q$. By Corollary \ref{Q'}, $(\nu,\Lambda_\mathbf{i})$ is a spectral pair, for all $\mathbf{i}\in I^\f$.
\end{proof}

\begin{corollary}\label{cor_P(0)}
Under Assumption *,  for each $\gamma\in \Gamma_0$, $P(\gamma)\ne\emptyset$.
\end{corollary}
\begin{proof}
Recall that  $P(\gamma)=\{\omega\in\Z:\gamma+N_1\omega\in\Lambda\}$ for every $\gamma\in [0,N_1)\cap\Q$. Since  $0\in\Lambda$, we have that  $P(0)\ne\emptyset$. By \eqref{def_pqg}, this implies that
\[
q_0(\xi)=\sum_{\omega\in P(0)}\Big|\hat{\nu}\bigg(\frac{\xi}{N_1}+\omega\bigg)\Big|^2>0,
\]
for $\xi=0$.
By \eqref{3}, we have
$$
q_0(0)=q_{\frac{N_1}{M_1}}(0)=\cdots=q_{\frac{(M_1-1)N_1}{M_1}}(0)>0,
$$
 and it follows that $P(\gamma)\ne\emptyset$ for all $\gamma\in \Gamma_0$.
\end{proof}

\section{Properties of spectral infinite convolutions} \label{sec_pfthm}
In this section, we study the properties of spectral infinite convolutions, and we give the proofs of Theorem \ref{consecutive digit set}, Theorem \ref{consecutive set} and Corollary \ref{f}.
	
\begin{proof}[Proof of Theorem~\ref{consecutive digit set}]
By Corollary \ref{prop}, without loss of generality, suppose that $(\mu,\Lambda)$ is a spectral pair with $0\in\Lambda$. For each $\mathbf{i}=(0,\gamma_1,\gamma_2,\cdots)$, by Theorem \ref{new spectrum}, $(\nu,\Lambda_\mathbf{i})$ is a spectral pair.
				
To show $M_2|N_2$, we decompose $\Lambda_{\mathbf{i}}$ by applying the similar approach to $\eqref{eqn_LDC}$, where we construct a spectrum for $\nu=\nu_{>1}$ by decomposing the spectrum $\Lambda$ of $\mu$ . First, we  write
$$
\Gamma'_0=\Big\{0,\frac{N_2}{M_2},\cdots,\frac{(M_2-1)N_2}{M_2}\Big\}.
$$
and
$$
P'(\gamma)=\{\omega\in\Z:\gamma+N_2\omega\in\Lambda_\mathbf{i}\},
$$
and it is clear that
$$
\Lambda_{\mathbf{i}}=\bigcup_{\gamma\in [0,N_2)\cap\Q}\{\gamma+N_2P'(\gamma)\}.
$$

Since $0\in\Lambda_{\mathbf{i}}$, by the same proof as Corollary \ref{cor_P(0)}, we have
$$
P'(\gamma)=\{\omega\in\Z:\gamma+N_2\omega\in\Lambda_\mathbf{i}\}\ne\emptyset,
$$
for all $\gamma\in\Gamma'_0$. Thus there exist $\omega_1, \omega_2,\cdots,\omega_{M_2-1}\in\Z$ such that, for all $k\in\{1,2,\cdots,M_2-1\}$, we have that
\[
\frac{kN_2}{M_2}+N_2\omega_k\in\Lambda_\mathbf{i}.
\]
On the other hand, for all $k\in\{1,2,\cdots,M_2-1\}$, by \eqref{def_Li}, this is equivalent to
\[
\frac{kN_2}{M_2}+N_2\omega_k\in\Lambda_\mathbf{i}=\bigcup_{j=0}^\f\bigg(\frac{\gamma_j}{N_1} +P(\gamma_j)\bigg),
\]
where $\gamma_j\in \Gamma$ given by \eqref{def_gam}, and $P(\cdot)$ given by \eqref{def_PG}.   Recall that $P(\gamma)\subset \Z$, for each $k\in\{1,2,\cdots,M_2-1\}$, there exists $\gamma_k'\in\Gamma$ such that
\[
\frac{kN_2}{M_2}-\frac{\gamma_k'}{N_1}\in\Z.
\]		

Suppose that   $M_2|N_2$ is not true, that is $\frac{N_2}{M_2}\notin\Z$. Let $m=\frac{M_2}{\gcd(N_2,M_2)}$. It is clear that $m\ge2$. Then for each $k\in\{1,2,\cdots,m-1\}$,  there exists $\gamma''_k\in \Gamma=[0,\frac{N_1}{M_1})\cap\Q$ such that
\[
\frac{k}{m}-\frac{\gamma''_k}{N_1}\in\Z,
\]
and it implies that
\begin{equation*}
\frac{k}{m}=\frac{\gamma''_k}{N_1}.
\end{equation*}
In particular, we have that
\begin{equation}\label{eqn_mkg}
\frac{m-1}{m}=\frac{\gamma''_{m-1}}{N_1}.
\end{equation}

However, since $m\geq 2$ and $\gamma''_k\in \Gamma=[0,\frac{N_1}{M_1})\cap\Q $, we have that
\[
\frac{m-1}{m}\ge \frac{1}{2}\ge\frac{N_1}{M_1}\cdot\frac{1}{N_1}>\frac{\gamma''_{m-1}}{N_1},
\]		
which contradicts ~\eqref{eqn_mkg}. Therefore, we obtain that $M_2|N_2$.

Since $\nu$ is also a spectral infinite convolution generated by complete residue systems $\{B_k\}_{k=2}^\infty$, we replace the spectral pair $(\mu, \Lambda)$ by $(\nu, \Lambda_\mathbf{i})$ and repeat the same process, and we have $M_3|N_3$. Therefore by induction, we have $M_k|N_k$ for all $k\geq 2$.
\end{proof}

\begin{proof}[Proof of Theorem \ref{consecutive set}]
If $M_k|N_k$, for all $k\ge2$, by Corollary \ref{change N}, without loss of generality, suppose that $M_k|N_k$, for all $k\in\N$. Thus $\{(N_k,B_k)\}$ is a sequence of admissible pair, by Theorem \ref{thm_MZ}, the infinite convolution $\mu$ is a spectral measure. We only need to prove the necessity.
   	
Since $B_k=\{0,1,2,\cdots,M_k-1\}$  and $N_k\ge M_k\ge2$, for all $k\in \N$, it is clear that
$f_{B_k}=1$ and $\nu_{>k}$ is supported on $[0,1]$ for every $k\in\N$. Immediately, this implies that $B_k$ satisfies {\it UDZ}, for all $k\in\N$ and $\{\nu_{>k}\}_{k=1}^\infty$ is tight.
Since the infinite convolution $\mu$ is a spectral measure, by Theorem \ref{consecutive digit set}, we have $M_k|N_k$, for all $k\ge2$.
\end{proof}

\begin{proof}[Proof of Corollary \ref{f}]
By Theorem \ref{consecutive digit set}, it is sufficient to show that $B_k$ satisfies {\it UDZ}, for all $k\in\N$.
	
Since
$$
f_{B_k}(x)=\sum_{i=0}^{2n_k}(-x)^i=\frac{1+x^{2n_k+1}}{1+x},
$$
it is clear that
$$
\mathcal{Z}(f_{B_k}\circ e^{-2\pi i x})\subseteq\bigg\{\frac{2j+1}{4n_k+2}:j\in\Z\bigg\}\subseteq\bigg\{\frac{j}{M_k}:j\in\Z\setminus M_k\Z\bigg\},
$$
which implies that $B_k$ satisfying {\it UDZ}, for all $k\in\N$.
\end{proof}

\section{Spectral measure supported on a non-compact set}\label{sec_noncpt}
Note that the tightness of $\{\nu_{>k}\}_{k=1}^\infty$ plays an important role in the spectrality of measures, but it is often hard to verify. In this section, we show that the tightness of $\{\nu_{>k}\}_{k=1}^\infty$  is guaranteed by the remainder bounded condition.

Recall that
a sequence $\{( N_k,B_k)\}_{k=1}^\infty $ satisfies remainder bounded condition if
\[\sum_{k=1}^{\infty} \frac{\# B_{k,2}}{\# B_k} < \infty,\]
where $B_{k,1}=B_k \cap \{0,1,\cdots,N_k-1\}$ and $B_{k,2}=B_k \backslash B_{k,1}.$

\begin{theorem}\label{existence}\cite{MZ24}
Let $\{( N_n,B_n)\}_{n=1}^\infty $ be a sequence satisfying remainder bounded condition. Then the infinite convolution $\mu$ exists.
\end{theorem}

\begin{lemma}\label{lem_tight}
Let $\mu$ and $\nu_{>k}$ be given by \eqref{infinite-convolution} and \eqref{nu_n} respectively with respect to $\{(N_k,B_k)\}_{k=1}^\infty $. If $\{(N_k,B_k)\}_{k=1}^\infty $ satisfies remainder bounded condition, then $\{\nu_{>k}\}_{k=1}^\infty$ is tight.
\end{lemma}
\begin{proof}
First, we prove the existence of $\{\nu_{>k}\}_{k=1}^\infty$.	
Since $\{(N_k,B_k)\}_{k=1}^\infty $ satisfies remainder bounded condition, we have
\[
\sum_{k=1}^{\infty} \frac{\# B_{k,2}}{\# B_k} < \infty,
\]
where $B_{k,1}=B_k \cap \{0,1,\cdots,N_k-1\}$ and $B_{k,2}=B_k \backslash B_{k,1}.$
For all $j\ge1$, since the sequence $\{(N_k,B_k)\}_{k=j+1}^\infty $ satisfies remainder bounded condition,  by Theorem \ref{existence}, the measure $\nu_{>j}$ exists.

Next, we show that $\{\nu_{>j}\}_{j=1}^\infty$ is tight, and the key is to estimate $\nu_{>j}([0,1])$. Recall that
$$
\nu_{>j}=\delta_{N_{j+1}^{-1} B_{j+1}} * \delta_{(N_{j+1} N_{j+2})^{-1} B_{j+2}} * \cdots,
$$
this implies that
\begin{align*}
\nu_{>j}([0,1])
&=\nu_{>j}\big(\big\{\sum_{k=j+1}^{\f}\frac{b_k}{N_{j+1}N_{j+2}\cdots N_k}:b_k\in B_k , \forall k>j\big\}\bigcap [0,1]\big)\\
&\ge1-\nu_{>j}\big(\bigcup_{m=j+1}^\f\big\{\sum_{k=j+1}^{m}\frac{b_k}{N_{j+1}N_{j+2}\cdots N_k}:b_{k}\in B_{k,1},\forall k<m , \ b_m\in B_{m,2}  \big\}\big)\\
&\ge1-\sum_{m=j+1}^{\f}\nu_{>j}\big(\big\{\sum_{k=j+1}^{m}\frac{b_k}{N_{j+1}N_{j+2}\cdots N_k}:b_{k}\in B_{k,1},\forall k<m , \ b_m\in B_{m,2}  \big\}\big)\\
&\ge1-\sum_{m=j+1}^{\f}\frac{\# B_{k,2}}{\# B_{j+1}\# B_{j+2 }\cdots\# B_{k} }\\
&\ge1-\sum_{m=j+1}^{\f}\frac{\# B_{k,2}}{\# B_k }.
\end{align*}
	
Since
\[
\sum_{k=1}^{\infty} \frac{\# B_{k,2}}{\# B_k} < \infty,
\]
for each $\epsilon>0$, there exists $J_\epsilon\in\N$ such that for all $j>J_\epsilon$,
\[
\sum_{k=j+1}^{\infty} \frac{\# B_{k,2}}{\# B_k} <\epsilon,
\]
and it implies that $\nu_{>j}([0,1])>1-\epsilon$, for all $j>J_\epsilon$.
	
On the other hand, for each $1<j\le J_\epsilon$, there exists a compact subset $K_j\subseteq\R$ such that
$$
\nu_{>j}(K_j)>1-\epsilon.
$$
Let $K=[0,1]\bigcup \Big (\bigcup_{j=0}^{J_\epsilon}K_j$ \Big).  Then we have
$$
\nu_{>j}(K)>1-\epsilon,
$$
for all $j\in\N$, that is ,   $\{\nu_{>j}\}_{j=1}^\infty$ is tight, and the conclusion holds.	
\end{proof}

Given $l\in (0,1)$, we say $\{( N_k,B_k)\}_{k=1}^\infty $ satisfies  \textit{partial concentration condition (PCC)}, if  one of the following two conditions holds

\noindent(i) there exists a sequence $\{(b_{k,1},b_{k,2}):b_{k,1},b_{k,2}\in[0,N_k-1]\}_{k=1}^\infty $ such that $0<\frac{b_{k,2}-b_{k,1}}{N_k}<l$ satisfying
\[
\sum_{k=1}^{\f}\frac{\# B_{k,2}^{l}}{\# B_k}<\f ,
\]
where $B_{k,1}^{l}=B_k\cap [b_{k,1},b_{k,2}]$ and $B_{k,2}^{l}=B_k\backslash B_{k,1}^{l};$

\noindent(ii) there exists $c\in(0,1]$ and a sequence $\{(b_{k,1},b_{k,2}):b_{k,1},b_{k,2}\in[\frac{l}{2} N_k,(1-\frac{l}{2})N_k]\}_{k=1}^\infty $ such that
\[
\frac{\# B_{k}^{l}}{\# B_k}\ge c,
\]
for all $k\in \N$, where $B_{k}^{l}=B_k\cap [b_{k,1},b_{k,2}].$\\

Next, we use the following theorem to prove the spectrality of measures that may not have compact support.

\begin{theorem}\label{gap-condition}\cite{MZ24}
	Let $\{( N_k,B_k)\}_{k=1}^\infty $ be a sequence of admissible pairs satisfying {\it RBC}.
	Suppose that a subsequence $\{( N_{m_k},B_{m_k})\}_{k=1}^\infty  $ satisfies  PCC.  Then the infinite convolution $\mu$ exists and  is a spectral measure.
\end{theorem}
\begin{theorem}\label{thm_necessary}
Given $\{(N_k,B_k)\}_{k=1}^\infty $ such that $B_k=\{0,1,2,\cdots,M_k-2,n_kM_k-1\}$ and $N_k\ge M_k\ge3$, $n_k\ge1$ are integers, for all $k\in \N$. Suppose that $\sum_{k=1}^{\f}\frac{1}{M_k}<\f$ and $\gcd(n_kM_k-1,M_k-2)=1$ for all $k$ such that $M_k$ is even and $M_k\ge5$. Then $\mu$ is a spectral measure if and only if $M_k|N_k$, for all $k\ge2$.
\end{theorem}
\begin{proof}
Since,
\[
\sum_{k=1}^{\infty} \frac{\# B_{k,2}}{\# B_k}=\sum_{k=1}^{\f}\frac{1}{M_k} < \infty,
\]
by Lemma \ref{lem_tight}, we know that $\{(N_k,B_k)\}$ satisfies RBC and $\{\nu_{>k}\}_{k=1}^\infty$ is tight.

First, we prove the sufficiency. Suppose that $M_k|N_k$, for all $k\ge2$.  By Corollary \ref{change N}, without loss of generality, we assume that $M_k|N_k$, for all $k\in\N$. Thus $\{(N_k,B_k)\}$ is a sequence of admissible pair, by Theorem \ref{gap-condition}, it is sufficient to show that there exists a subsequence of $\{(N_k,B_k)\}$ satisfies $PCC$.

For $k\in\N$, either $\frac{M_k}{N_k}\ge\frac{1}{2}$ appears infinitely many times, or $\frac{M_k}{N_k}<\frac{1}{2}$ appears infinitely many times.

For $\#\{k\in\N:\frac{M_k}{N_k}\ge\frac{1}{2}\}=\f$, without loss of generality, we assume that $\{(N_{m_k},B_{m_k})\}$ is a subsequence of $\{(N_k,B_k)\}$ and such that $\frac{M_{m_k}}{N_{m_k}}\ge\frac{1}{2}$, for all $k\in\N$.  By choosing  $l=\frac{1}{2}$ and $(b_{m_k,1},b_{m_k,2})=(\frac{N_{m_k}}{4},\frac{3N_{m_k}}{4})$, it is clear that
\[
\frac{\# B_{m_k}^{l}}{\# B_{m_k}}\ge \frac{1}{8},
\]
for all $k\in \N$, where $B_{m_k}^{l}=B_{m_k}\cap [\frac{N_{m_k}}{4},\frac{3N_{m_k}}{4}]. $ This implies that $\{(N_{m_k},B_{m_k})\}$ satisfies PCC.

For  $\#\{k\in\N:\frac{M_k}{N_k}<\frac{1}{2}\}=\f$, without loss of generality, we assume that $\{(N_{t_k},B_{t_k})\}$ is a subsequence of $\{(N_k,B_k)\}$ and such that $\frac{M_{t_k}}{N_{t_k}}<\frac{1}{2}$, for all $k\in\N$. By choose $l=\frac{1}{2}$ and $(b_{t_k,1},b_{t_k,2})=(0,\frac{N_{t_k}}{2})$, it is clear that $B_{t_k,2}^{l}\subseteq\{n_{t_k}M_{t_k}-1\}$, for all $k\in\N$, thus
\[
\sum_{k=1}^{\f}\frac{\# B_{t_k,2}^{l}}{\# B_{t_k}}\le\sum_{k=1}^{\f}\frac{1}{M_{t_k}}\le\sum_{k=1}^{\f}\frac{1}{M_k}<\f. 
\]
Hence $\{(N_{t_k},B_{t_k})\}$ satisfies PCC,  and there exists a subsequence of $\{(N_k,B_k)\}$ satisfies PCC, by Theorem \ref{gap-condition}, the infinite convolution $\mu$ exists and  is a spectral measure.

Next, we prove the necessity. Since $\mu$ is a spectral measure and $B_k$ is a  complete residue system modulo $M_k$, for all $k\in \N$, by Theorem \ref{gap-condition}, it is sufficient to show that $B_k$ satisfies uniform discrete zero condition for every $k\geq 1$.

Suppose that $n_k=1$ for all $k\in \N$, i.e., $B_k=\{0,1,2,\cdots,M_k-1\}$  and $N_k\ge M_k\ge2$. It is clear that
$$
f_{B_k}=1.
$$
 Immediately, we have that $\mathcal{Z}(f_{B_k}\circ e^{-2\pi i x})=\emptyset$, and it implies that $B_k$ satisfies {\it UDZ}, for all $k\in\N$. Otherwise we assume that $n_k\ge2$ for some $k\in \N$. We set
\begin{equation*}\label{eq_g}
g_{B_k}(x)=1+x+\cdots+x^{M_k-2}+x^{M_k-1+n_kM_k}=\frac{1-x^{M_k-1}}{1-x}+x^{n_kM_k-1}.
\end{equation*}
If $x_0\in\mathcal{Z}(g_{B_k})$ and $|x_0|=1$, we have
$$
1-x_0^{M_k-1}=(x_0-1)x_0^{n_kM_k-1}, \ |x_0|=1 , \text{and} \ x_0\ne1,
$$
which implies that
$$
\left\{
\begin{array}{rl}
	1&=x_0^{n_kM_k},\\
	x_0^{M_k-1}&=x_0^{n_kM_k-1},\\
	x_0&\ne1,
\end{array}\right.
\qquad \textit{or} \qquad
\left\{
\begin{array}{rl}
	1&=-x_0^{n_kM_k-1},\\
	-x_0^{M_k-1}&=x_0^{n_kM_k},\\
	x_0&\ne1.
\end{array}\right.
$$
Hence there exists $j_1, j_2\in\Z$ such that
$$
\left\{
\begin{array}{rl}
	x_0&=e^{2\pi i \frac{j_1}{n_kM_k}},\\
	x_0&=e^{2\pi i \frac{j_2}{(n_k-1)M_k}}\ \text{or}\ 0,\\
	x_0&\ne1,
\end{array}\right.
\qquad \textit{or} \qquad
\left\{
\begin{array}{rl}
	x_0&=e^{2\pi i \frac{2j_1+1}{2n_kM_k-2}},\\
	x_0&=e^{2\pi i \frac{2j_2+1}{2(n_k-1)M_k+2}}\ \text{or}\ 0,\\
	x_0&\ne1.
\end{array}\right.
$$
Therefore, we have
$$
\frac{j_1}{n_kM_k}-\frac{j_2}{(n_k-1)M_k}\in\Z\
\qquad \textit{or} \qquad  \frac{2j_1+1}{2n_kM_k-2}-\frac{2j_2+1}{2(n_k-1)M_k+2}\in\Z.
$$
Without loss of generality, we assume that $j_1,j_2\in\Z$ satisfying
$$
\frac{j_1}{n_kM_k}=\frac{j_2}{(n_k-1)M_k} \qquad \textit{or} \qquad    \frac{2j_1+1}{2n_kM_k-2}=\frac{2j_2+1}{2(n_k-1)M_k+2},
$$
which implies that
$$
\ n_k|j_1\  \text{or}\ 2j_2+1=\frac{(j_1-j_2)(2(n_k-1)M_k+2)}{M_k-2}.
$$
Thus there exists $j\in\Z$ satisfying $j\nmid M_k$ and $j\nmid M_k-2$ such that
$$
x_0=e^{2\pi i \frac{j}{M_k}} \ \text{or} \ e^{2\pi i \frac{j}{M_k-2}}.
$$

For $M_k=3, 4$, obviously, we have
$$
\mathcal{Z}(g_{B_k}\circ e^{-2\pi i x})\subseteq\bigg\{\frac{j}{M_k}:j\in\Z\setminus M_k\Z\bigg\}.
$$

For $M_k\ge5$, if there exists $j\in\Z\setminus (M_k-2)\Z$ such that
$$
g_{B_k}(e^{2\pi i \frac{j}{M_k-2}})=\sum_{t=0}^{M_k-2}e^{2\pi i \frac{tj}{M_k-2}}+e^{2\pi i \frac{(n_kM_k-1)j}{M_k-2}}=1+e^{2\pi i \frac{(n_kM_k-1)j}{M_k-2}}=0,
$$
then we have
$2\cdot\frac{(n_kM_k-1)j}{M_k-2}$ is odd, which implies that $M_k-2$ is even.
Thus, if $M_k$ is odd, we have
$$
\mathcal{Z}(g_{B_k}\circ e^{-2\pi i x})\subseteq\bigg\{\frac{j}{M_k}:j\in\Z\setminus M_k\Z\bigg\}.
$$

If $M_k$ is even, since $\gcd(n_kM_k-1,M_k-2)=1$, we have $2\cdot\frac{j}{M_k-2}$ is odd,
which implies that
$$
\mathcal{Z}(g_{B_k}\circ e^{-2\pi i x})\subseteq\bigg\{\frac{j}{M_k}:j\in\Z\setminus M_k\Z\bigg\}\cup\bigg\{\frac{j}{2}:j\in\Z\setminus 2\Z\bigg\}=\bigg\{\frac{j}{M_k}:j\in\Z\setminus M_k\Z\bigg\}.
$$
Since
$$
g_{B_k}(x)=f_{B_k}(x)\sum_{k=0}^{M-1}x^k,
$$
we have
$$
\mathcal{Z}(f_{B_k}\circ e^{-2\pi i x})\subseteq\mathcal{Z}(g_{B_k}\circ e^{-2\pi i x})\subseteq\bigg\{\frac{j}{M_k}:j\in\Z\setminus M_k\Z\bigg\},
$$
which means that
$B_k$ satisfies uniform discrete zero condition for every $k\geq 1$.
Then, if the infinite convolution $\mu$ is a spectral measure, by Theorem \ref{consecutive digit set}, we have $M_k|N_k$, for all $k\ge2$.
\end{proof}

\begin{proof}[Proof of Theorem~\ref{thm_consecutive}]
The sufficiency  and necessity is directly from Theorem \ref{thm_necessary}.
\end{proof}

 \section*{Acknowledgments}
The initial research question was suggested by Prof. He Xingang, and the authors wish to thank Prof. Xinggang He and Prof. Lixiang An for their helpful comments during the research.

\end{document}